\documentclass[12pt]{amsart}
\usepackage{amsmath,amssymb,amsbsy,amsfonts,latexsym,amsopn,amstext,cite,
                                               amsxtra,euscript,amscd,bm}
\usepackage{url}
\usepackage[colorlinks,linkcolor=blue,anchorcolor=blue,citecolor=blue]{hyperref}
\usepackage{color}
\usepackage{graphics,epsfig}
\usepackage{graphicx}
\usepackage{float}

\usepackage{mathtools}
\usepackage{todonotes}

\begin{document}
\newtheorem{problem}{Problem}
\newtheorem{theorem}{Theorem}
\newtheorem{lemma}[theorem]{Lemma}
\newtheorem{claim}[theorem]{Claim}
\newtheorem{cor}[theorem]{Corollary}
\newtheorem{prop}[theorem]{Proposition}
\newtheorem{definition}{Definition}
\newtheorem{question}[theorem]{Question}

\numberwithin{equation}{section}
\numberwithin{theorem}{section}

\def\cA{{\mathcal A}}
\def\cB{{\mathcal B}}
\def\cC{{\mathcal C}}
\def\cD{{\mathcal D}}
\def\cE{{\mathcal E}}
\def\cF{{\mathcal F}}
\def\cG{{\mathcal G}}
\def\cH{{\mathcal H}}
\def\cI{{\mathcal I}}
\def\cJ{{\mathcal J}}
\def\cK{{\mathcal K}}
\def\cL{{\mathcal L}}
\def\cM{{\mathcal M}}
\def\cN{{\mathcal N}}
\def\cO{{\mathcal O}}
\def\cP{{\mathcal P}}
\def\cQ{{\mathcal Q}}
\def\cR{{\mathcal R}}
\def\cS{{\mathcal S}}
\def\cT{{\mathcal T}}
\def\cU{{\mathcal U}}
\def\cV{{\mathcal V}}
\def\cW{{\mathcal W}}
\def\cX{{\mathcal X}}
\def\cY{{\mathcal Y}}
\def\cZ{{\mathcal Z}}

\def\A{{\mathbb A}}
\def\B{{\mathbb B}}
\def\C{{\mathbb C}}
\def\D{{\mathbb D}}
\def\E{{\mathbb E}}
\def\F{{\mathbb F}}
\def\G{{\mathbb G}}
\def\I{{\mathbb I}}
\def\J{{\mathbb J}}
\def\K{{\mathbb K}}
\def\L{{\mathbb L}}
\def\M{{\mathbb M}}
\def\N{{\mathbb N}}
\def\O{{\mathbb O}}
\def\P{{\mathbb P}}
\def\Q{{\mathbb Q}}
\def\R{{\mathbb R}}
\def\S{{\mathbb S}}
\def\T{{\mathbb T}}
\def\U{{\mathbb U}}
\def\V{{\mathbb V}}
\def\W{{\mathbb W}}
\def\X{{\mathbb X}}
\def\Y{{\mathbb Y}}
\def\Z{{\mathbb Z}}

\def\ep{{\mathbf{e}}_p}
\def\em{{\mathbf{e}}_m}
\def\eq{{\mathbf{e}}_q}

\def\scr{\scriptstyle}
\def\\{\cr}
\def\({\left(}
\def\){\right)}
\def\[{\left[}
\def\]{\right]}
\def\<{\langle}
\def\>{\rangle}
\def\fl#1{\left\lfloor#1\right\rfloor}
\def\rf#1{\left\lceil#1\right\rceil}
\def\le{\leqslant}
\def\ge{\geqslant}
\def\eps{\varepsilon}
\def\mand{\qquad\mbox{and}\qquad}

\def\sssum{\mathop{\sum\ \sum\ \sum}}
\def\ssum{\mathop{\sum\, \sum}}
\def\ssumw{\mathop{\sum\qquad \sum}}

\def\vec#1{\mathbf{#1}}
\def\inv#1{\overline{#1}}
\def\num#1{\mathrm{num}(#1)}
\def\dist{\mathrm{dist}}

\def\fA{{\mathfrak A}}
\def\fB{{\mathfrak B}}
\def\fC{{\mathfrak C}}
\def\fU{{\mathfrak U}}
\def\fV{{\mathfrak V}}

\newcommand{\bflambda}{{\boldsymbol{\lambda}}}
\newcommand{\bfxi}{{\boldsymbol{\xi}}}
\newcommand{\bfrho}{{\boldsymbol{\rho}}}
\newcommand{\bfnu}{{\boldsymbol{\nu}}}

\def\GL{\mathrm{GL}}
\def\SL{\mathrm{SL}}

\def\Hba{\overline{\cH}_{a,m}}
\def\Hta{\widetilde{\cH}_{a,m}}
\def\Hb1{\overline{\cH}_{m}}
\def\Ht1{\widetilde{\cH}_{m}}

\def\flp#1{{\left\langle#1\right\rangle}_p}
\def\flm#1{{\left\langle#1\right\rangle}_m}
\def\dmod#1#2{\left\|#1\right\|_{#2}}
\def\dmodq#1{\left\|#1\right\|_q}

\def\Zm{\Z/m\Z}

\def\Err{{\mathbf{E}}}

\newcommand{\comm}[1]{\marginpar{%
\vskip-\baselineskip 
\raggedright\footnotesize
\itshape\hrule\smallskip#1\par\smallskip\hrule}}

\def\xxx{\vskip5pt\hrule\vskip5pt}


\title{On the distribution of fractions with power denominator}

 \author[B. Kerr] {Bryce Kerr}
\address{School of Physical, Environmental and Mathematical Sciences, The University of New South Wales Canberra, Australia}
\email{b.kerr@adfa.edu.au}

\begin{abstract} 
In this paper we obtain a sharp upper bound for the number of solutions to a certain diophantine inequality involving fractions with power denominator. This problem is motivated by a conjecture of Zhao concerning the spacing of such fractions in short intervals and the large sieve for power modulus. As applications of our estimate we show Zhao's conjecture is true except for a set of small measure and give a new $\ell_1 \rightarrow \ell_2$ large sieve inequality for power modulus. 
\end{abstract}

\maketitle
\section{Introduction}
Given integers $k$ and $N$ consider the set of fractions of the form
\begin{align}
\label{eq:SkN}
S_k(N)=\left\{ \frac{u}{n^{k}}, \quad 1\le u\le n^{k}, \quad 1\le n \le N, \quad (n,u)=1 \right\}.
\end{align}
In this paper we consider the distribution of $S_k(N)\subset [0,1]$. This problem first appears to be considered by Zhao~\cite{Zhao1} and is motivated by large sieve type inequalities for characters to power moduli. For such inequalities, one seeks to determine the smalllest $\Delta_k(N,M)$ depending on $k,N$ and $M$ such that for any sequence of complex numbers $\alpha_m$ we have
\begin{align*}
\sum_{1\le n \le N}\sum_{\substack{a=1 \\ (a,n)=1}}^{n^{k}}\left|\sum_{K<m \le K+M}\alpha_me\left(\frac{am}{n^{k}} \right) \right|^2\le \Delta_k(N,M)\|\alpha\|_2^{2}.
\end{align*}
We note the classical large sieve inequality states that 
$$\Delta_1(N,M)\ll M+N^{2},$$
and this implies the estimates 
\begin{align}
\label{eq:classicallargesieve}
\Delta_k(N,M)\ll M+N^{2k} \quad \text{and} \quad \Delta_k(N,M)\ll NM+N^{k+1}.
\end{align}
In~\cite{Zhao1}, Zhao provided some estimates for $\Delta_k(N,M)$ which go beyond~\eqref{eq:classicallargesieve} and conjectured that
\begin{align}
\label{eq:conj1}
\Delta_k(N,M)\ll (N^{k+1}+M)N^{o(1)},
\end{align}
which is based on heuristics for the density of points of the form~\eqref{eq:SkN}. One may establish quantitative relationships between $\Delta_k(N,M)$ and estimates for the number of points of $S_k(N)$ in short intervals. For real numbers $x$ and $Y$ we define
\begin{align}
\label{eq:IxY}
I_{k,N}(x,Y)=\left|\left\{ z\in S_k(N) \ : \ \left\|z-x\right\|\le \frac{1}{Y} \right\}\right|.
\end{align}
In~\cite{Zhao1}, Zhao conjectured that for any $x$ we have 
\begin{align}
\label{eq:conj2}
I_{k,N}\left(x,N^{k+1}\right)\ll N^{o(1)},
\end{align}
and showed that~\eqref{eq:conj2} implies~\eqref{eq:conj1}. There have been a number of improvements and extensions of the results of~\cite{Zhao1} to which we refer the reader to~\cite{BZ0,BZ11,Hal} for power moduli, to~\cite{Bai,Hal1,Hal2,Hal3} for more general sparse sets of moduli and to~\cite{Bai1,BB,BS} for extensions to function fields and Gaussian integers. Such large sieve inequalities have had applications to a variety of different areas of number theory to which we refer the reader to~\cite{BZ01,BZ2,BZ3,Mat} for applications to the distribution of primes, to~\cite{BPS,ShpZhao} for elliptic curves and to~\cite{BBS,BFKS,Shp1} for arithmetic problems in finite fields and function fields.

In this paper we consider the distribution of points in $S_k(N)$. For integers $k,N,Y$ we let $I_k(N,Y)$ count the number of solutions to the inequality 
\begin{align*}
\left|\frac{u_1}{n_1^{k}}-\frac{u_2}{n_2^{k}}\right|\le \frac{1}{Y},
\end{align*}
with variables satisfying
\begin{align*}
1\le n_1,n_2\le N, \quad 1\le u_1\le n_1^{k}, \quad 1\le u_2 \le n_2^{k},
\end{align*}
and note that $I_k(N,Y)$ may be considered as an $\ell_2$-norm estimate of $I_{k,N}(x,Y)$ defined in~\eqref{eq:IxY}. We give a sharp estimate for $I_k(N,Y)$  and in particular show that 
\begin{align*}
I_k(N,Y)\le \left(\frac{N^{2k+2}}{Y}+N^{k+1}\right)N^{o(1)}.
\end{align*}
As applications of our estimate we show Zhao's conjecture~\eqref{eq:conj2} is true except for a set of small measure and provide a new large sieve type inequality for fractions with power denominator which may be thought of as an $\ell_1\rightarrow \ell_2$ version of the conjecture~\eqref{eq:conj1}. 
\subsection{Acknowledgement}
The author would like to thank Lee Zhao for a number of useful conversations.
\section{Main results}

\begin{theorem}
\label{thm:main1}
For integers $k$ and $N$  and a positive real number $Y$ we let $I_k(N,Y)$ count the number of solutions to the inequality 
\begin{align*}
\left|\frac{u_1}{n_1^k}-\frac{u_2}{n_2^k}\right|\le \frac{1}{Y},
\end{align*}
with variables satisfying 
$$1\le n_1,n_2\le N, \quad 1\le u_1\le n_1^k, \quad 1\le u_2 \le n_2^k.$$
Then we have 
\begin{align*}
I_k(N,Y)\ll \left(\frac{N^{2k+2}}{Y}+N^{k+1}\right)N^{o(1)}.
\end{align*}
\end{theorem}
As an immediate Corollary to Theorem~\ref{thm:main1}, we have.
\begin{cor}
Let $k$ and $N$ be integers, $Y$  and $x$  real numbers  and define $I_{k,N}(x,Y)$   as in~\eqref{eq:IxY}. For any $\varepsilon>0$ have 
\begin{align*}
\mu\left( x\in [0,1] \ : \ I_{k,N}(x,N^{k+1})\ge N^{\varepsilon} \right)\le N^{-\varepsilon+o(1)},
\end{align*}
where $\mu$ denotes the Lebesgue measure.
\end{cor}

Another consequence of Theorem~\ref{thm:main1} is the following $\ell_1\rightarrow \ell_2$ large sieve inequality.
\begin{cor}
\label{cor2}
For any integers $k,M,N$ and any sequence of complex numbers $\alpha_m$ we have 
\begin{align*}
\sum_{1\le n \le N}\sum_{\substack{a=1 \\ (a,n)=1}}^{n^{k}}\left|\sum_{K<m\le K+M}\alpha_me\left(\frac{am}{n^k}\right) \right|\ll (N^{k+1}+M^{1/2}N^{(k+1)/2})N^{o(1)}\|\alpha\|_2.
\end{align*}
\end{cor}
We note that the conjecture~\eqref{eq:conj1} implies Corollary~\ref{cor2} by the Cauchy-Schwarz inequality and that Corollary~\ref{cor2} is sharper than what one would obtain using results of~\cite{BZ0,BZ11,Hal}.

\section{Preliminary results}
The following forms the basis of the van der Corput method of exponential sums, for a proof  see~\cite[Theorem~8.16]{IwKo}.

\begin{lemma}
\label{lem:vdc}
For any real valued function $f$ defined on an interval $[a,b]$ with derivatives satisfying
\begin{align*}
\frac{T}{M^2} \ll f''(z) \ll \frac{T}{M^2}, \quad |f^{(3)}(z)|\ll \frac{T}{M^3}, \quad |f^{(4)}(z)|\le \frac{T}{M^4},\quad z\in[a,b],
\end{align*}
we have
\begin{align*}
\sum_{a<n<b}e(f(n))=\sum_{\alpha<m<\beta}f''(x_m)^{-1/2}e\left(f(x_m)-mx_m+\frac{1}{8} \right)+E,
\end{align*}
where $\alpha=f'(a), \beta=f'(b)$, $x_m$ is the unique solution to $f'(x)=m$ for $x\in [a,b]$ and
\begin{align*}
E&\ll \frac{M}{T^{1/2}}+\log(|f'(b)-f'(a)|+2).
\end{align*}

\end{lemma}
Our main application of Lemma~\ref{lem:vdc} will be when $f$ is a monomial and in this case we have the following.
\begin{lemma}
\label{lem:vdc2}
Let  $N,\eta, \alpha,$ and $y$ be real numbers satisfying
$$\alpha\not \in \N, \quad 1\le \eta\ll1, \quad \rho>0, \quad y>0.$$
Then we have 
\begin{align*}
& \sum_{N<n<\eta N}e\left(\frac{y}{\alpha}\left(\frac{n}{N} \right)^{\alpha}\right)= \\ & ((\beta-1)y)^{1/2}\sum_{c_1M<m<c_2M}\frac{1}{m}\left(\frac{m}{M}\right)^{\beta/2}e\left(\frac{1}{8}-\frac{y}{\beta}\left(\frac{m}{M} \right)^{\beta} \right)+O\left(\frac{N}{y^{1/2}}+\log{y} \right),
\end{align*}
where $M$ and $\beta$ are defined by
$$M=\frac{y}{N}, \quad \frac{1}{\alpha}+\frac{1}{\beta}=1,$$
and the constants $c_1$ and $c_2$ depend only on $\alpha$ and $\eta$.
\end{lemma}

We note that the constants $c_1$ and $c_2$ occuring in Lemma~\ref{lem:vdc2} can be given explicitly in terms of $\alpha$ although to state this result requires seperating cases $\alpha>0$ and $\alpha<0$ and we  have found the above form more convenient.

The following is a special case of~\cite[Lemma~2.1]{W}.
\begin{lemma}
\label{lem:eqntomv}
For two intervals $I_1,I_2$, a real valued function  
$$\phi:I_1\times I_2 \rightarrow \R,$$
and a positive real number $Y,$ we let $J(I_1,I_2,\phi,Y)$ count the number of solutions to the inequality 
$$|\phi(u_1,n_1)-\phi(u_2,n_2)|\le \frac{1}{Y},$$
with integer variables 
$$n_1,n_2\in I_1, \quad u_1,u_2\in I_2.$$
We have 
\begin{align*}
J(I_1,I_2,\phi,Y)\ll \frac{1}{Y}\int_{-Y}^{Y}\left|\sum_{n\in I_1}\sum_{u\in I_2}e(y\phi(n,u))\right|^2dy,
\end{align*}
and for any sequence of complex numbers $\theta(n,u)$ satisfying 
\begin{align}
\label{eq:thetacond}
|\theta(n,u)|\le 1,
\end{align}
 we have
\begin{align*}
\frac{1}{Y}\int_{-Y}^{Y}\left|\sum_{n\in I_1}\sum_{u\in I_2}\theta(n,u)e(y\phi(n,u))\right|^2dy\ll J\left(I_1,I_2,\phi,Y\right).
\end{align*}
\end{lemma}

The next two results are consequences of Lemma~\ref{lem:eqntomv}.
\begin{lemma}
\label{lem:mvshorter}
For intervals $I_1$ and $I_2,$ a real valued function 
$$\phi:I_1\times I_2 \rightarrow \R,$$
and positive real numbers $Y$ and $Z$ with $Z\le Y$ we have 
\begin{align*}
\frac{1}{Y}\int_{-Y}^{Y}\left|\sum_{u\in I_1}\sum_{n\in I_2}e(y\phi(n,u))\right|^2dy\ll \frac{1}{Z}\int_{-Z}^{Z}\left|\sum_{u\in I_1}\sum_{n\in I_2}e(y\phi(n,u))\right|^2dy.
\end{align*}
\end{lemma}

\begin{lemma}
\label{lem:mvcoefficients}
For intervals $I_1$ and $I_2,$ a real valued function 
$$\phi:I_1\times I_2 \rightarrow \R,$$
a positive real number $Y$ and a sequence of complex numbers $\theta(n,u)$ satisfying 
$$|\theta(n,u)|\le 1,$$
we have  
\begin{align*}
\int_{-Y}^{Y}\left|\sum_{u\in I_1}\sum_{n\in I_2}\theta(n,u)e(y\phi(n,u))\right|^2dy\ll \int_{-Y}^{Y}\left|\sum_{u\in I_1}\sum_{n\in I_2}e(y\phi(n,u))\right|^2dy.
\end{align*}
\end{lemma}

The following is~\cite[Lemma~6]{FoI}.
\begin{lemma}
\label{lem:mellin}
Let $V,L,\nu$ and $\lambda$ be real numbers satisfying
\begin{align*}
0<L\le V\le \nu V<\lambda L,
\end{align*}
and let $a_{v}$ be a sequence of complex numbers satisfying $|a_v|\le 1$. Then we have 
\begin{align*}
\sum_{V<v<\nu V}a_v=\frac{1}{2\pi}\int_{-L}^{L}\left(\sum_{L<\ell <\lambda L}a_{\ell}\ell^{-it} \right)V^{it}(\nu^{it}-1)t^{-1}dt+O(\log{(2+L)}).
\end{align*}
\end{lemma}
\begin{lemma}
\label{lem:vdctransformfinal}
Let $Y,N,U,\alpha,\rho,\gamma$ and $\eta$ be positive real numbers satisfying
\begin{align*}
Y\ge 2, \quad \alpha \not \in \N, \quad \eta,\rho>1.
\end{align*}
There exists real numbers $Y_1,Y_2,c_1,c_2$  satisfying
$$Y\ll Y_1 \ll Y_2\ll Y, \quad  1 \ll c_3 \ll c_4 \ll 1,$$
and a sequence of complex numbers $\theta(n,u)$ satisfying 
$$|\theta(n,u)|\le 1,$$
 such that
\begin{align*}
& \int_{Y}^{2Y}\left|\sum_{N\le n< \eta N}\sum_{U\le u< \rho U}e\left(y\left(\frac{u}{U}\right)^{\alpha}\left(\frac{n}{N}\right)^{\gamma}\right) \right|^2dy\ll \\ &  \frac{U^2(\log{Y})^2}{Y}\int_{Y_1}^{Y_2}\left|\sum_{N\le n <\eta N}\sum_{c_3V\le v\le c_4V}\theta(n,v)e\left(y\left(\frac{n}{N}\right)^{\gamma(1-\beta)}\left(\frac{v}{V}\right)^{\beta} \right)\right|^2dy \\ &  +N^2(\log{Y})^2(U^2+Y),
\end{align*}
with implied constants depending only on $\alpha,\rho,\gamma$ and $\eta$ and $\beta$ and $V$ are is given by 
$$\frac{1}{\alpha}+\frac{1}{\beta}=1, \quad V=\frac{Y}{U}.$$
\end{lemma}
\begin{proof}
By Lemma~\ref{lem:vdc2}, for fixed $Y\le y \le 2Y$ and $N\le n <\eta N,$ we have 
\begin{align*}
& \sum_{U\le u< \rho U}e\left(y\left(\frac{u}{U}\right)^{\alpha}\left(\frac{n}{N}\right)^{\gamma}\right)=\left(\frac{(\beta-1)\alpha yn^{\gamma}}{N^{\gamma}}\right)^{1/2} \\ & \times\sum_{c_1yn^{\gamma}U^{-1}N^{-\gamma}<v<c_2yn^{\gamma}U^{-1}N^{-\gamma}}\frac{1}{v}\left(\frac{v}{yn^{\gamma}U^{-1}N^{-\gamma}}\right)^{\beta/2}e\left(\frac{1}{8}-\frac{\alpha yn^{\gamma}}{N^{\gamma}\beta}\left(\frac{v}{yn^{\gamma}U^{-1}N^{-\gamma}} \right)^{\beta} \right) \\ &  \quad \quad \quad +O\left(\log{Y}\left(\frac{U}{Y^{1/2}}+1\right)\right),
\end{align*}
where $\beta$ is given by 
\begin{align*}
\frac{1}{\alpha}+\frac{1}{\beta}=1,
\end{align*}
and the constants $c_1$ and $c_2$ depend only on $\alpha$. Summing the above over $N\le n<\eta N$ and taking absolute values, we see that
\begin{align*}
&\left|\sum_{N\le n<\eta N}\sum_{U\le u< \rho U}e\left(y\left(\frac{u}{U}\right)^{\alpha}\left(\frac{n}{N}\right)^{\gamma}\right)\right|\ll \frac{U}{Y^{1/2}} \\ & \left|\sum_{N\le n <\eta N}\sum_{c_1yn^{\gamma}V/YN^{\gamma}<v<c_2yn^{\gamma}V/YN^{\gamma}}\theta(n,v)e\left(\frac{\alpha y}{\beta}\left(\frac{Y}{y}\right)^{\beta}\left(\frac{n}{N}\right)^{\gamma(1-\beta)}\left(\frac{v}{V}\right)^{\beta} \right) \right| \\ & \quad \quad \quad \quad \quad +N\log{Y}\left(1+\frac{U}{Y^{1/2}}\right),
\end{align*}
where 

\begin{align}
\label{eq:Vdef}
V=\frac{Y}{U},
\end{align}
and
\begin{align*}
\theta(n,v)=\left(\frac{n}{N}\right)^{\gamma/2}\frac{Y}{vU}\left(\frac{v}{Yn^{\gamma}U^{-1}N^{-\gamma}}\right)^{\beta/2},
\end{align*}
so that if $V\ll v \ll V$ and $N\le n <\eta N$ we have 
\begin{align*}
|\theta(n,v)|\ll 1.
\end{align*}
Defining
\begin{align*}
&W(y)= \\ & \left|\sum_{N\le n <\eta N}\sum_{c_1yn^{\gamma}V/YN^{\gamma}<v<c_2yn^{\gamma}V/YN^{\gamma}}\theta(n,v)e\left(\frac{\alpha y}{\beta}\left(\frac{Y}{y}\right)^{\beta}\left(\frac{n}{N}\right)^{\gamma(1-\beta)}\left(\frac{v}{V}\right)^{\beta} \right) \right|,
\end{align*}
the above simplifies to
\begin{align}
\label{eq:Psstep1-5}
&\left|\sum_{N\le n<\eta N}\sum_{U\le u< \rho U}e\left(y\left(\frac{u}{U}\right)^{\alpha}\left(\frac{n}{N}\right)^{\gamma}\right)\right|\ll \frac{U}{Y^{1/2}}W(y)+N\log{Y}\left(1+\frac{U}{Y^{1/2}}\right).
\end{align}

We next remove the dependence on $n$ and $y$ in summation over $v$ in $W(y)$. There exists absolute constants $c_3$ and $c_4$ depending only on $c_1,c_2$ and $\eta$ such that whenever $Y\le y \le 2Y$ and $N\le n \le \eta N$ we have 
\begin{align*}
\left[\frac{c_1 y n^{\gamma}V}{YN^{\gamma}},\frac{c_2 y n^{\gamma}V}{YN^{\gamma}}\right]\subseteq [c_3V,c_4V],
\end{align*}
and hence by Lemma~\ref{lem:mellin}, for some $c$ depending only on $c_1$ and $c_2$
\begin{align*}
W(y)&=\frac{1}{2\pi}\int_{-V}^{V}\left(\sum_{N\le n <\eta N}\sum_{c_3V\le v\le c_4V}\theta(n,v)e\left(\frac{\alpha y}{\beta}\left(\frac{Y}{y}\right)^{\beta}\left(\frac{n}{N}\right)^{\gamma(1-\beta)}\left(\frac{v}{V}\right)^{\beta} \right)v^{-it}\right) \\ & \times \left(\frac{yn^{\gamma}V}{YN^{\gamma}}\right)^{it}(c^{it}-1)t^{-1}dt+O\left(N(\log{(V+2)})\right).
\end{align*}
Taking absolute values and applying the triangle inequality gives
\begin{align*}
|W(y)| & \ll \int_{-V}^{V}\left|\sum_{N\le n <\eta N}\sum_{c_3V\le v\le c_4V}\theta(n,v,t)e\left(\frac{\alpha y}{\beta}\left(\frac{Y}{y}\right)^{\beta}\left(\frac{n}{N}\right)^{\gamma(1-\beta)}\left(\frac{v}{V}\right)^{\beta} \right)\right|G(t)dt
\\ & \quad \quad \quad \quad +O\left(N(\log{Y})\right),
\end{align*}
where $G(t)$ is the unique continuous function on $\R$ defined for nonzero $t$ by
\begin{align*}
G(t)=\frac{|c^{it}-1|}{|t|},
\end{align*}
so that 
\begin{align*}
\int_{-V}^{V}G(t)dt\ll \log{(V+2)} \ll \log{Y},
\end{align*}
and $\theta(n,v,t)$ is given by 
\begin{align*}
\theta(n,v,t)=\theta(n,v)v^{-it}n^{it\gamma}.
\end{align*}
This implies that for some $|t_0|\le V$ we have 
\begin{align*}
& |W(y)|\ll \\ & \log{Y}\left|\sum_{N\le n <\eta N}\sum_{c_3V\le v\le c_4V}\theta(n,v,t_0)e\left(\frac{\alpha y}{\beta}\left(\frac{Y}{y}\right)^{\beta}\left(\frac{n}{N}\right)^{\gamma(1-\beta)}\left(\frac{v}{V}\right)^{\beta} \right)\right| \\ & \quad \quad \quad \quad \quad +N(\log{Y}),
\end{align*}
and hence by~\eqref{eq:Psstep1-5}
\begin{align}
\label{eq:transformfinal}
\int_{Y}^{2Y}\left|\sum_{N\le n<\eta N}\sum_{U\le u< \rho U}e\left(y\left(\frac{u}{U}\right)^{\alpha}\left(\frac{n}{N}\right)^{\gamma}\right)\right|^2dy  &\ll  \frac{U^2(\log{Y})^2}{Y}T(Y) \\ & \quad \quad +N^2(U^2+Y)(\log{Y})^2, \nonumber
\end{align}
where 
$$T(Y)=\int_{Y}^{2Y}\left|\sum_{N\le n <\eta N}\sum_{c_3V\le v\le c_4V}\theta(n,v,t_0)e\left(\frac{\alpha y}{\beta}\left(\frac{Y}{y}\right)^{\beta}\left(\frac{n}{N}\right)^{\gamma(1-\beta)}\left(\frac{v}{V}\right)^{\beta} \right)\right|^2dy.$$

The change of variable
\begin{align*}
z=\frac{\alpha Y^{\beta}y^{1-\beta}}{\beta},
\end{align*}
in the above integral implies that there exists $Y_1$ and $Y_2$ satisfying
\begin{align*}
Y\ll Y_1 \ll Y_2 \ll Y,
\end{align*}
such that 
\begin{align*}
T(Y)\ll \int_{Y_1}^{Y_2}\left|\sum_{N\le n <\eta N}\sum_{c_3V\le v\le c_4V}\theta(n,v,t_0)e\left(z\left(\frac{n}{N}\right)^{\gamma(1-\beta)}\left(\frac{v}{V}\right)^{\beta} \right)\right|^2dz,
\end{align*}
and the result follows from~\eqref{eq:transformfinal}.
\end{proof}

\begin{lemma}
\label{lem:Jcs1}
For integers $U_1,N_1,U_2,N_2,Y$ and $k$ we let $J_k(U_1,N_1,U_2,N_2,Y)$ count the number of solutions to the inequality 
\begin{align*}
\left|\frac{u_1}{n_1^k}-\frac{u_2}{n_2^k} \right|\le \frac{1}{Y},
\end{align*} 
with integer variables satisfying
\begin{align*}
U_1 \le u_1< 2U_1, \ \ N_1 \le n_1< 2N_1, \ \ U_2\le u_2< 2U_2, \ \ N_2\le n_2< 2N_2,
\end{align*}
and when $U_1=U_2$ and $N_1=N_2$ we write
\begin{align*}
J_k(U_1,N_1,U_2,N_2,Y)=J_k(U_1,N_1,Y).
\end{align*}
We have 
\begin{align*}
J_k(U_1,N_1,U_2,N_2,Y)\ll J_k(U_1,N_1,Y)^{1/2}J_k(U_2,N_2,Y)^{1/2}.
\end{align*}
\end{lemma}
\begin{proof}
For integer $j$ let $J_1(j)$ count the number of solutions to the inequality
\begin{align}
\label{eq:J1def}
\frac{j}{Y}\le \frac{u}{n^k}< \frac{j+1}{Y},
\end{align}
with variables satisfying
\begin{align*}
U_1\le u_1<2U_1, \quad N_1\le n_1<2N_1,
\end{align*}
and let $J_2(j)$ count the number of solutions to the inequality~\eqref{eq:J1def} with variables satisfying
\begin{align*}
U_2\le u_1<2U_2, \quad N_2\le n_1<2N_2.
\end{align*}
We have 
\begin{align*}
J_k(U_1,N_1,U_2,N_2,Y)\le \sum_{\substack{j_1,j_2 \\ |j_1-j_2|\le 1}}J_1(j_1)J_2(j_2),
\end{align*}
and hence by the Cauchy-Schwarz inequality 
\begin{align*}
J_k(U_1,N_1,U_2,N_2,Y)^2&\ll \left(\sum_{j}J_1(j)^2\right)\left(\sum_{j}J_2(j)^2\right) \\ & \ll J_k(U_1,N_1,Y)J_k(U_2,N_2,Y).
\end{align*}
\end{proof}
The following is known as the Kusmin-Landau inequality, for a proof see~\cite[Corollary~8.11]{IwKo}.
\begin{lemma}
\label{lem:kusminlandau}
Let $f$ be a continuously differentiable function on some interval $I$ with derivative satisfying
$$\|f'(x)\|\ge \lambda, \quad x\in I.$$
Then we have 
\begin{align*}
\sum_{n\in I}f(n)\ll \lambda^{-1}.
\end{align*}
\end{lemma}
\begin{lemma}
\label{lem:J*bound}
For integers $k,N,U$ and a real number $W$, if 
\begin{align}
\label{eq:Wcond}
W\ll U,
\end{align}
for a sufficiently small constant then
\begin{align*}
\int_{W}^{2W}\left|\sum_{N\le n < 2N}\sum_{U\le u<2U}e\left(z\left(\frac{n}{N}\right)^k\left(\frac{U}{u}\right)\right) \right|^2dy\ll       
\frac{U^2N^2}{W}.
\end{align*}
\end{lemma}
\begin{proof}
 By Lemma~\ref{lem:kusminlandau} and~\eqref{eq:Wcond}, for each $W \le z \le 2W$ and $N\le n \le 2N$ we have 
\begin{align*}
\sum_{U\le u<2U}e\left(z\left(\frac{n}{N}\right)^k\left(\frac{U}{u}\right)\right)\ll \frac{U}{W},
\end{align*}
and hence 
\begin{align*}
\int_{W}^{W}\left|\sum_{N\le n < 2N}\sum_{U\le u<2U}e\left(z\left(\frac{n}{N}\right)^k\left(\frac{U}{u}\right)\right) \right|^2dy\ll \frac{U^2N^2}{W}.
\end{align*}
\end{proof}
\begin{lemma}
\label{lem:smallY}
For integers $k,N,M,U$ and $V$ satisfying
\begin{align}
\label{eq:smallYcond}
M\ll N, \quad V\ll U, \quad U\ll N^{k},
\end{align}
and a real number  $Y$ satisfying 
\begin{align}
\label{eq:smallYcond1}
Y\le  \frac{N^{k}}{2},
\end{align} 
 we have 
\begin{align*}
\frac{1}{Y}\int_{-Y}^{Y}\left|\sum_{N<n\le M}\sum_{U<u\le V}e\left(\frac{yu}{n^k}\right) \right|^2dy\ll \frac{N^2U^2}{Y}+\frac{N^{2k+2}(\log{N})^2}{Y}.
\end{align*}
\end{lemma}
\begin{proof}
Bounding the contribution from $y\in [-1,1]$ in the above integral trivially gives
\begin{align*}
\frac{1}{Y}\int_{-Y}^{Y}\left|\sum_{N<n\le M}\sum_{U<u\le V}e\left(\frac{yu}{n^k}\right) \right|^2dy\le \frac{N^2U^2}{Y}+\frac{1}{Y}\int_{1}^{Y}\left|\sum_{N<n\le M}\sum_{U<u\le V}e\left(\frac{yu}{n^k}\right) \right|^2dy.
\end{align*}
For any fixed $1\le y\le Y$ we have 
\begin{align*}
\sum_{N<n\le M}\sum_{U<u\le V}e\left(\frac{yu}{n^k}\right)\ll \sum_{N<n\le M}\min\left(U,\frac{1}{\|y/n^k\|}\right),
\end{align*}
and for any $N<n\le M$ and $1\le y \le Y,$ by~\eqref{eq:smallYcond1} we have 
\begin{align*}
\frac{y}{n^k}\le \frac{1}{2},
\end{align*}
and for any $N<n_1<n_2\le M$  
\begin{align*}
 \frac{y}{n_1^k}-\frac{y}{n_2^k}\gg \frac{y(n_2-n_1)}{N^{k+1}},
\end{align*}
so that
\begin{align*}
\sum_{N<n\le M}\min\left(U,\frac{1}{\|y/n^k\|}\right)\ll \frac{N^{k+1}}{y}\sum_{1\le n\le M}\frac{1}{n}\ll \frac{N^{k+1}(\log{N})}{y},
\end{align*}
and hence 
\begin{align*}
\frac{1}{Y}\int_{1}^{Y}\left|\sum_{N<n\le M}\sum_{U<u\le V}e\left(\frac{yu}{n^k}\right) \right|^2dy\ll \frac{N^{2k+2}(\log{N})^2}{Y},
\end{align*}
from which the desired result follows.
\end{proof}

\begin{lemma}
\label{lem:firstcase1}
Let $N$ be an integer, $Y$ and $\varepsilon$ positive real numbers satisfying
\begin{align}
\label{eq:firstcase1ycond}
Y\le N^{k+\varepsilon}.
\end{align}
 Then with $I_k(N,Y)$  as in Theorem~\ref{thm:main1}, we have
\begin{align*}
I_k(N,Y)\ll  \frac{N^{2k+2+\varepsilon+o(1)}}{Y}
\end{align*}
\end{lemma}
\begin{proof}
With notation as in Lemma~\ref{lem:Jcs1}, we have 
\begin{align*}
I_k(N,Y)&\le \sum_{\substack{i,j \\ 2^{i}\le 2N \\ 2^j\le 2^{ki+1}}}\sum_{\substack{\ell,s \\ 2^{\ell}\le 2N \\ 2^s \le 2^{k\ell+1}}}J_k(2^j,2^i,2^s,2^\ell,Y) \\ 
&\le \left(\sum_{\substack{i,j \\ 2^{i}\le 2N \\ 2^j\le 2^{ki+1}}}J_k(2^j,2^s,Y)^{1/2}\right)^2,
\end{align*}
and hence 
\begin{align}
\label{eq:IkJk1}
I_k(N,Y) &\ll N^{o(1)}\max_{\substack{M\le 2N \\ V\le 2M^{k}}} J_k(M,V,Y).
\end{align}
Fix some $M\le 2N$ and $V\le 2N^k$ and consider $J_k(M,V,Y)$. By Lemma~\ref{lem:eqntomv} we have 
\begin{align*}
J_k(M,V,Y)\ll \frac{1}{Y}\int_{-Y}^{Y}\left|\sum_{M<n\le 2M}\sum_{V\le v <2V}e\left(\frac{yu}{n^k}\right) \right|^2dy.
\end{align*}
If $Y<M^k/2$ then by Lemma~\ref{lem:smallY}
\begin{align}
\label{eq:Jkcase1}
J_k(M,V,Y)\ll \frac{M^2V^2}{Y}+\frac{N^{2k+2}(\log{N})^2}{Y}\ll \frac{N^{2k+2}(\log{N})^2}{Y},
\end{align}
and if $M^k/2<Y$ then by Lemma~\ref{lem:mvshorter} and Lemma~\ref{lem:smallY}
\begin{align*}
J_k(M,V,Y)&\ll \frac{1}{M^k}\int_{-M^k/4}^{M^k/4}\left|\sum_{M<n\le 2M}\sum_{V\le v <2V}e\left(\frac{yu}{n^k}\right) \right|^2dy \\ 
&\ll \frac{M^2V^2}{M^k}+(\log{M})^2M^{k+2}\ll (\log{M})^2M^{k+2}.
\end{align*}
By~\eqref{eq:firstcase1ycond} we have 
\begin{align*}
M^{k+2}\le \frac{YM^{k+2}}{Y}\le \frac{N^{2k+2+\varepsilon}}{Y},
\end{align*}
which gives 
\begin{align*}
J_k(M,V,Y)\ll \frac{N^{2k+2+\varepsilon}(\log{N})^2}{Y}.
\end{align*}
Combining the above with~\eqref{eq:IkJk1}  and~\eqref{eq:Jkcase1} we get
\begin{align*}
I_k(N,Y)\ll \frac{N^{2k+2+\varepsilon}(\log{N})^4}{Y},
\end{align*}
and completes the proof.
\end{proof}
\section{Proof of Theorem~\ref{thm:main1}}
By Lemma~\ref{lem:firstcase1} we may suppose 
\begin{align}
\label{eq:Ythm1ub}
Y\ge N^{k+\varepsilon},
\end{align}
for a sufficiently small $\varepsilon$ and using estimates for the divisor function, we may suppose 
$$Y\le N^{2k},$$
which will be used to estimate terms $Y^{o(1)}=N^{o(1)}$ in what follows.
 Arguing as in the proof of Lemma~\ref{lem:firstcase1}, we have 

\begin{align}
\label{eq:IkJk11}
I_k(N,Y) &\ll N^{o(1)}J_k(M,U,Y),
\end{align}
for some 
\begin{align*}
M\le 2N, \quad U\le 2M^k.
\end{align*}
If $n_1,n_2,u_1,u_2$ satisfy 
\begin{align*}
\left|\frac{u_1}{n_1^k}-\frac{u_2}{n_2^k} \right|\le \frac{1}{Y},
\end{align*}
and
\begin{align}
\label{eq:varcond111}
M\le n_1,n_2\le 2M, \quad U\le u_1,u_2\le 2U,
\end{align}
then we have 
\begin{align}
\label{eq:J*def}
\left|\left(\frac{n_1}{M}\right)^k\left(\frac{U}{u_1}\right)-\left(\frac{n_2}{M}\right)^k\left(\frac{U}{u_2}\right)\right|\le \frac{M^k}{UY}.
\end{align}
Hence defining
\begin{align}
\label{eq:Zdef}
Z=\frac{UY}{M^k}
\end{align}
and letting $J^{*}(M,U,Z)$ count the number of solutions to the inequality~\eqref{eq:J*def} with variables satisfying~\eqref{eq:varcond111}, we have 
\begin{align*}
I_k(N,Y)\ll N^{o(1)}J^{*}(M,U,Z).
\end{align*}
By Lemma~\ref{lem:eqntomv}
\begin{align*}
I_k(N,Y)\ll \frac{N^{o(1)}}{Z}\int_{0}^{Z}\left|\sum_{\substack{M\le m\le 2M \\ U\le u \le 2U}}e\left(z\left(\frac{n}{M}\right)^k\left(\frac{U}{u}\right)\right) \right|^2dz,
\end{align*}
and performing a dyadic partition of the above integral, there exists some $1\le X\le Z$ such that 
\begin{align}
\label{eq:IWX}
I_k(N,Y)\ll \frac{N^{o(1)}M^2U^2}{Z}+\frac{N^{o(1)}}{Z}W(X),
\end{align}
where 
\begin{align*}
W(X)=\int_{X}^{2X}\left|\sum_{\substack{M\le m\le 2M \\ U\le u \le 2U}}e\left(z\left(\frac{n}{M}\right)^k\left(\frac{U}{u}\right)\right) \right|^2dz.
\end{align*}
If $X\ll U$ then we apply Lemma~\ref{lem:J*bound} and use~\eqref{eq:Zdef} to get
\begin{align*}
I_k(N,Y)\ll \frac{N^{o(1)}M^2U^2}{Z}\ll \frac{N^{o(1)}M^2UN^k}{Y}\ll \frac{N^{2k+2+o(1)}}{Y}.
\end{align*}
Hence it remains to consider when $X\gg U$.  An application  of Lemma~\ref{lem:vdctransformfinal} to $W(X)$ results in
\begin{align*}
& W(X)\ll \\ &  \frac{U^2(\log{Y})^2}{X}\int_{c_3X}^{c_4X}\left|\sum_{M\le n <2M}\sum_{c_3V\le v\le c_4V}\theta(n,v)e\left(y\left(\frac{n}{M}\right)^{k/2}\left(\frac{v}{V}\right)^{1/2} \right)\right|^2dy \\ &  +M^2(\log{Y})^2(U^2+X),
\end{align*}
where $c_1,\dots,c_4$ are absolute constants and 
\begin{align}
\label{eq:Vdef}
V=\frac{X}{U}.
\end{align}
Combining with~\eqref{eq:IWX} and extending the range of integration to $[0,c_4X]$ we get 
\begin{align}
\label{eq:Ikstep222}
&I_k(N,Y)\ll \frac{N^{2k+2+o(1)}}{Y}+N^{2+o(1)} \\ &+\frac{U^2N^{o(1)}}{Z}\frac{1}{X}\int_{0}^{c_4X}\left|\sum_{M\le n <2M}\sum_{c_3V\le v\le c_4V}\theta(n,v)e\left(y\left(\frac{n}{M}\right)^{k/2}\left(\frac{v}{V}\right)^{1/2} \right)\right|^2dy.
\end{align}
Hence defining $J_0$ to count the number of solutions to the inequality 
\begin{align}
\label{eq:J0eqn}
\left|n_1^{k/2}v_1^{1/2}-n_2^{k/2}v_2^{1/2} \right|\ll \frac{M^{k/2}V^{1/2}}{X}, 
\end{align}
with variables satisfying
\begin{align}
\label{eq:J0cond}
M\le n_1,n_2 \le 2M, \quad V\ll u_1,u_2\ll V,
\end{align}
using~\eqref{eq:Ikstep222} and Lemma~\ref{lem:eqntomv} we get 
\begin{align}
\label{eq:Ikstep333}
I_k(N,Y)\ll \frac{N^{2k+2+o(1)}}{Y}+N^{2+o(1)}+\frac{U^2N^{o(1)}}{Z}J_0.
\end{align}
If $n_1,n_2,u_1,u_2$ satisfy~\eqref{eq:J0eqn} and~\eqref{eq:J0cond} then 
\begin{align*}
|n_1^kv_1-n_2^kv_2|\ll \frac{M^kV}{X} \ll \frac{M^k}{U}.
\end{align*}
Choosing $n_1,v_1$ with $O(MV)$ possibilities gives at most
$$O\left(1+\frac{M^k}{U}\right),$$
choices for the value of $n_2^kv_2$ and hence estimates for the divisor function imply that 
\begin{align*}
J_0\ll M^{1+o(1)}V\left(1+\frac{M^k}{U}\right).
\end{align*}
Combining with~\eqref{eq:Ikstep333} gives 
\begin{align*}
I_k(N,Y)\ll \frac{N^{2k+2+o(1)}}{Y}+N^{2+o(1)}+\frac{U^2VM^{1+o(1)}}{Z}\left(1+\frac{M^k}{U}\right).
\end{align*}
Recalling~\eqref{eq:Zdef},~\eqref{eq:Vdef} and $X\ll Z$ the above simplifies to 
\begin{align*}
I_k(N,Y)&\ll \frac{N^{2k+2+o(1)}}{Y}+N^{2+o(1)}+UM^{1+o(1)}\left(1+\frac{M^k}{U}\right) \\
&\ll \frac{N^{2k+2+o(1)}}{Y}+UM^{1+o(1)}+M^{k+1+o(1)} \\ & \ll \frac{N^{2k+2+o(1)}}{Y}+N^{k+1+o(1)},
\end{align*}
and completes the proof.
\section{Proof of Corollary~\ref{cor2}}
We first recall a special case of the duality principle to which we refer the reader to~\cite[Lemma~2]{Mont} for a proof.
\begin{lemma}
\label{lem:duality}
Let $c_{n,m}$ be a sequence of complex numbers. Suppose that for every sequence of complex numbers $\alpha_n$ we have 
\begin{align*}
\left(\sum_{m}\left|\sum_{n}\alpha_n c_{n,m} \right|^2\right)^{1/2}\le D\|\alpha\|_{\infty}.
\end{align*}
Then for any sequence of complex numbers $\alpha_m$ we have 
\begin{align*}
\sum_{n}\left|\sum_{m}\alpha_m c_{n,m} \right|\le D\|\alpha\|_2.
\end{align*}
\end{lemma}
By Lemma~\ref{lem:duality}, in order to prove Corollary~\ref{cor2} it is sufficient to show that for any sequence of complex numbers $\alpha_{a,n}$ satisfying
\begin{align*}
|\alpha_{a,n}|\le A,
\end{align*}
we have 
\begin{align}
\label{eq:sss111222}
\sum_{K<m\le K+M}\left|\sum_{\substack{1\le n \le N \\ 1\le a \le n^{k} \\ (a,n)=1}}\alpha_{a,n}e\left(\frac{am}{n^{k}} \right) \right|^2\ll \left(N^{2k+2}+MN^{k+1}\right)N^{o(1)}A^2.
\end{align}
Let $\phi$ be a positive valued Schwartz function satisfying
\begin{align*}
\phi(x)\ge 1, \quad |x|\le 1 \quad \text{and} \quad \text{supp}(\widehat \phi)\subseteq [-2,2],
\end{align*}
where $\widehat \phi$ denotes the Fourier transform. We have 
\begin{align*}
\sum_{K<m\le K+M}\left|\sum_{\substack{1\le n \le N \\ 1\le a \le n^{k} \\ (a,n)=1}}\alpha_{a,n}e\left(\frac{am}{n^{k}} \right) \right|^2\le \sum_{m\in \Z}\phi\left(\frac{m-K}{M}\right)\left|\sum_{\substack{1\le n \le N \\ 1\le a \le n^{k} \\ (a,n)=1}}\alpha_{a,n}e\left(\frac{am}{n^{k}} \right) \right|^2,
\end{align*}
which after expanding the square and interchanging summation gives 
\begin{align*}
&\sum_{K<m\le K+M}\left|\sum_{\substack{1\le n \le N \\ 1\le a \le n^{k} \\ (a,n)=1}}\alpha_{a,n}e\left(\frac{am}{n^{k}} \right) \right|^2\le \\ 
& A^2\sum_{\substack{1\le n_1,n_2 \le N \\ 1\le a_1 \le n_1^{k} \\ 1\le a_2 \le n_2^{k}}}\left|\sum_{m\in \Z}\phi\left(\frac{m-K}{M}\right)e\left(\left(\frac{a_1}{n_1^{k}}-\frac{a_2}{n_2^{k}}\right)m \right) \right|.
\end{align*}
Applying Poission summation and using the fact that $\text{supp}(\widehat \phi)\subseteq [-2,2]$ we see that
\begin{align*}
\sum_{m\in \Z}\phi\left(\frac{m-K}{M}\right)e\left(\left(\frac{a_1}{n_1^{k}}-\frac{a_2}{n_2^{k}}\right)m \right)\ll M \quad \text{if} \quad \left\|\frac{a_1}{n_1^{k}}-\frac{a_2}{n_2^{k}}\right\|\le \frac{4}{M},
\end{align*}
and
\begin{align*}
\sum_{m\in \Z}\phi\left(\frac{m-K}{M}\right)e\left(\left(\frac{a_1}{n_1^{k}}-\frac{a_2}{n_2^{k}}\right)m \right)=0 \quad \text{otherwise},
\end{align*}
and hence with notation as in Theorem~\ref{thm:main1}
\begin{align*}
\sum_{K<m\le K+M}\left|\sum_{\substack{1\le n \le N \\ 1\le a \le n^{k} \\ (a,n)=1}}\alpha_{a,n}e\left(\frac{am}{n^{k}} \right) \right|^2\ll A^2 MI_k(N,4M)\ll A^2\left( N^{2k+2}+MN^{k+1}\right)N^{o(1)},
\end{align*}
which establishes~\eqref{eq:sss111222} and completes the proof.


\begin{thebibliography}{99}
\bibitem{Bai}
S. Baier,
{\it On the large sieve with sparse sets of moduli}, J. Ramanujan Math. Soc. {\bf 21} (2006), no. 3, 279--295.

\bibitem{Bai1}
S. Baier, {\it The large sieve with square norm moduli in $\Z[i]$}, J. Theor. Nombres Bordeaux, {\bf 30}, (2018), no. 1, 93--115.

\bibitem{BB}
S. Baier and A. Bansal, {\it The large sieve with power moduli for $\Z[i]$}, Int. J. Number Theory, {\bf 14}, (2018), no. 10, 2737--2756.

\bibitem{BBS}
S. Baier, A. Bansal and K. Singh, {\it Divisibility problems for function fields}, arXiv:1803.07457.

\bibitem{BS}
S. Baier and R. K. Singh, {\it Large sieve inequality with power moduli for function fields}, arXiv:1802.03131.

\bibitem{BZ0}
S. Baier and L. Zhao, {\it Large sieve inequality with characters for powerful moduli},
Int. J. Number Theory,
{\bf 1}
(2005) 265--280.

\bibitem{BZ01}
S. Baier and L. Zhao, {\it Bombieri-Vinogradov type theorems for sparse sets of moduli}, Acta Arith. {\bf 125}, (2006), no. 2, 187--201.

\bibitem{BZ1}
S. Baier and L. Zhao,
{\it Primes in quadratic progressions on average}, Math. Ann. {\bf 338},
(2007), no. 4, 963--982.

\bibitem{BZ11}
S. Baier and L. Zhao,
{\it An improvement for the large sieve for square moduli,}
 J. Number Theory
{\bf 128} (2008), no. 1, 154--174.

\bibitem{BZ2}
S. Baier and L. Zhao, {\it On primes represented by quadratic polynomials}, Anatomy of integers, (2008),  159--166.
\bibitem{BZ3}
S. Baier and L. Zhao, {\it On primes in quadratic progressions}, Int. J. Number Theory {\bf 5},
(2009), no. 6, 1017--1035.

\bibitem{BPS}
W. D. Banks, F. Pappalardi and I. E. Shparlinski. {\it On group structures realized by elliptic curves over arbitrary finite fields},
Exp. Math., {\bf 21}, (1), (2012), 11--25.




\bibitem{BFKS}
J.  Bourgain,  K.  Ford, S.  V.  Konyagin,  I.  E.  Shparlinski,
{\it On the Divisibility of Fermat Quotients},
Michigan Math. J.
{\bf 59}
(2010), 313--328.



\bibitem{FoI}
E. Fouvry and H. Iwaniec, {\it Exponential sums with monomials,} J. Number Th. 33 (1989), 311--333.

\bibitem{IwKo} H. Iwaniec and E. Kowalski, {\it Analytic Number Theory}, Colloquium Publications {\bf 53} (2004), American Math. Soc., Providence, RI.

\bibitem{Hal}
K. Halupczok, {\it A new bound for the large sieve inequality with power moduli},  Int. J. Number
Theory {\bf 8} (2012), no. 3, 689--695.

\bibitem{Hal1}
K. Halupczok, {\it Large sieve inequalities with general polynomial moduli}, Q. J.
Math. {\bf 66} (2015), no. 2, 529--545.

\bibitem{Hal2}
K. Halupczok, {\it Vinogradov's Mean Value Theorem as an Ingredient in Polynomial Large Sieve Inequalities and Some Consequences}, Irregularities in the Distribution of Prime Numbers, 97--109.

\bibitem{Hal3}
K. Halupczok, {\it Bounds for discrete moments of Weyl sums and applications}, arXiv:1804.05587.


\bibitem{Mat}
K. Matom\"{a}ki, {\it A note on primes of the form $p=aq^2+1$}, Acta. Arith. {\bf 137}, (2009), no. 2, 133--137.

\bibitem{Mont}
H. L. Montgomery, {\it The analytic principle of the large sieve}, Bull. Amer. Math. Soc. {\bf 84} (1978), 547--567.



\bibitem{Shp1}
 I. E. Shparlinski, {\it Fermat quotients:  exponential sums, value set and primitive roots},
Bull. Lond. Math. Soc.
{\bf 43},
(2011), 1228--1238.

\bibitem{ShpZhao}
I. E. Shparlinski and L. Zhao, {\it Elliptic curves in Isogeny classes}, J. Number Theory, {\bf 191}, (2018), 194--212.


\bibitem{W}
N. Watt, {\it Exponential sums and the Riemann zeta-function II}, J. London Math. Soc.
{\bf 39}
(1989), 385--404.

\bibitem{Zhao1}
L. Zhao, {\it Large sieve inequality for characters to square moduli},
Acta Arith.,
{\bf 112},
(2004), 297--308.



\end{thebibliography}
\end{document}